\documentclass[10pt]{article}

\usepackage[colorlinks,linkcolor=blue,citecolor=blue]{hyperref}
\usepackage{tabulary} 
\usepackage{amsthm}
\usepackage{multirow}
\usepackage{marginnote}
\usepackage{a4wide}
\usepackage{amssymb}
\usepackage{amsfonts}
\usepackage{amsmath}
\usepackage{mathrsfs}
\usepackage{tikz}
\usetikzlibrary{arrows,matrix}
\usetikzlibrary{positioning}
\usepackage{mdframed} 
\usepackage{lipsum} 
\usepackage{extarrows} 
\usepackage{color}
\usepackage{enumitem} 
\usepackage{tocloft} 
\usepackage{titlesec} 

\usepackage[all]{xy} 

\input xy
\xyoption{arrow} \xyoption{matrix}

\date{}

\newtheorem{proposition}{Proposition}[section]
\newtheorem{theorem}[proposition]{Theorem}
\newtheorem{lemma}[proposition]{Lemma}

\newtheorem{definition}[proposition]{Definition}
\newtheorem{corollary}[proposition]{Corollary}

\def\der{\partial }

\def\nFM0{{\nu }_{F,M_0}}
\def\nFN0{{\nu }_{F,N_0}}
\def\nGN0{{\nu }_{G,N_0}}

\def\N0{ {\bf N}_0 }

\def\t{\otimes}

\def\ra{\rightarrow}

\def\Xpm{X^{\pm }}

\def\s{\sigma}
\def\Z{\mathbb{Z}}

\def\l1{{\lambda}_1}

\def\a{\alpha}
\def\a0{ {\alpha }_0}
\def\a1{ {\alpha }_1}

\def\l{\lambda}

\def\nFGM0{{\nu }_{F,G,M_0}}

\def\nFN0{{\nu}_{F,N_0}}

\def\sm{{\sigma}^m}

\def\sm1{{\sigma}^{-1}}

\def\smtp1{{\sigma}^{-t+1}}

\def\S1{S^{-1}}

\def\Xpm1{X^{\pm 1}_1}

\def\sPM1{{\sigma }^{\pm 1}}
\def\sMP1{{\sigma }^{\mp 1 }}

\def\b{\beta}
\def\d{\delta}

\def\di{{\rm d.ind}}

\def\L{\Lambda}
\def\O{\Omega}

\def\G{\Gamma}

\def\CA{{\cal A}}

\def\CD{{\cal D}}

\def\Ytm1{Y^{t-1}}
\def\Yim1{Y^{i-1}}

\def\CL{{\cal L}}

\def\CN{{\cal N}}

\def\CF{{\cal F}}
\def\CG{{\cal G}}

\def\Aut{{\rm Aut}}
\def\bK{\overline{K}}

\def\Der{{\rm Der }}

\def\dim{{\rm dim }}
\def\char{{\rm char }}

\def\SL2Z{ {\rm SL}_2({\bf Z}) }

\def\CL{{\cal L}}

\def\Gp1{ G^{1 , 1 } }
\def\P11{ P^{-1 , 1 } }
\def\Pp1{ P^{1 , 1 } }

\def\nCLsr{{}^\nu\kern-2pt {\cal L}^{\sigma , \rho  }}
\def\nP{{}^\nu \kern-2pt P}
\def\nL{{}^\nu\kern-2pt L}
\def\nLL{{}^\nu\kern-2pt \Lambda}
\def\nPsr{{}^\nu\kern-2pt P^{\sigma , \rho  }}
\def\nLsr{{}^\nu\kern-2pt L^{\sigma , \rho  }}
\def\nuCL{{}^\nu\kern-2pt  {\cal L}}
\def\nCLsr{{}^\nu\kern-2pt {\cal L}^{\sigma , \rho  }}
\def\nCL1m{{}^\nu\kern-2pt {\cal L}^{-1 , 1  }}
\def\x1nu{x^\frac{1}{\nu}}
\def\xm1nu{x^{-\frac{1}{\nu}}}

\def\CN{{\cal N}}
\def\ra{\rightarrow }

\def\CB{{\cal B}}

\def\nAM0{{\nu }_{{\cal A},M_0}}
\def\nAN0{{\nu }_{{\cal A},N_0}}

\def\End{ {\rm End }}
\def\Der{ {\rm Der }}

\def\SL{{\rm SL}}

\def\di!{\frac{\der^i}{i!}}
\def\dik!{\frac{\der^k_i}{k!}}

\def\Fp{\mathbb{F}_p}

\def\gl{\mathfrak{l}}

\def\Max{{\rm Max}}

\def\N{\mathbb{N}}

\def\0{\overline{0}}
\def\1{\overline{1}}

\def\Ln1{\L_{n,\overline{1}}}

\def\a1{a_{\overline{1}}}

\def\S{\Sigma}

\def\vn1{\overrightarrow{n-1}}

\def\gl{{\rm gl}}
\def\sl{{\rm sl}}

\def\mL{\mathbb{L}}

\def\gf{\mathfrak{f}}

\def\mJ{\mathbb{J}}
\def\mI{\mathbb{I}}

\def\mF{\mathbb{F}}

\def\K1{{\rm K}_1}

\def\hmI1{\widehat{\mI_1}}
\def\tmI1{\widetilde{\mI_1}}
\def\tmJ1{\widetilde{\mJ_1}}
\def\hB1{\widehat{B_1}}
\def\hCB1{\widehat{\CB_1}}

\def \S{\mathcal{S}}

\def\sl2{\mathfrak{sl}_2}

\def\Deg{{\rm Deg}}

\def\sl2{\mathfrak{sl}_2}
\def\gl2{\mathfrak{gl}_2}

\def\b1{\overline{1}}

\def\fCK{{\mathfrak{C}}(K)}

\def\fCdK{{\mathfrak{C}}_d(K)}

\def\gl{{\mathfrak{l}}}

\makeatletter
\newenvironment{proof*}[1][\proofname]{\par
  \pushQED{\qed}%
  \normalfont \partopsep=\z@skip \topsep=\z@skip
  \trivlist
  \item[\hskip\labelsep
        \itshape
    #1\@addpunct{.}]\ignorespaces
}{%
  \popQED\endtrivlist\@endpefalse
}
\makeatother

\begin{document}

\author{V. V. \  Bavula 
}

\title{Analogue of the Galois Theory for arbitrary finite field extensions 
}

\maketitle

\begin{abstract}

This paper resolves the long-standing open problem of establishing a unified Galois theory for all finite field extensions. This paper is a finishing touch to the (over 200 years) {\em classical} `Galois Theory' of {\em arbitrary}  finite field extensions, i.e. the goal of it is to describe intermediate subfields of an arbitrary finite field extension  via {\em invariants} of `natural/obvious' objects that  are associated with subfields via two Galois-type correspondences.  

The classical Galois Theory covers the case of  finite Galois field extensions.  For finite  Galois field extensions the objects are their  Galois groups and their invariants. 

 In \cite{GaloisTh-RingThAp},  
we introduce a new  (ring theoretic) approach to the Galois Theory which is based on
 the  {\em principle  of maximal symmetry}.   In \cite{AnGaloisTh-NORMAL-Fields}, the  maximal symmetry of {\em normal} finite field extensions yields an analogue of the  Galois Theory for them. For a normal  finite field extension $L/K$ the  `natural/obvious' objects are  the subalgebra $\CD (L/K)\rtimes G(L/K)$ of $\End (L/K)$ that is generated by the automorphism group $G(L/K)$ and the algebra $\CD (L/K)$ of differential operators on $L/K$ and their  `invariants'.  The `maximal symmetry' means the equality  $\End (L/K)=\CD (L/K)\rtimes G(L/K)$  which turns out to be  a characteristic property of  {\em normal} finite field extensions, \cite{AnGaloisTh-NORMAL-Fields}.
 
 In the Galois theory,  the two Galois Correspondences are given in terms of the fields $L^H$ of $H$-{\em invariants} where $H$ is a subgroup of the Galois group $G(L/K)$. But for normal field extensions,  analogues of the two Galois Correspondences are given in terms of the fields  $L^H\cap L^{\CG}$ of $H$-invariants and $\CG$-{\em invariants/constants} where $\CG$ is a dominant Lie algebra and $(\CG , H)$ is a dom-group.

  The aim of this paper is  to obtain an analogue of the Galois Theory for {\em arbitrary} finite field extensions based on  the group-Lie algebra-theoretic approach, results and ideas of  \cite{GaloisTh-RingThAp} and \cite{AnGaloisTh-NORMAL-Fields}.

$\noindent$



{ \small \tableofcontents}

\end{abstract}


\section{Introduction} \label{INTROD} 

In this paper, module means a {\em left} module.  The following
notation will remain fixed throughout the paper (if it is not
stated otherwise): 
 
 \begin{itemize}
 

\item $K$ is a  field and $\bK$ is its algebraic closure,  

\item If  the field $K$ has  characteristic $p>0$,  the map  $\gf :\bK\ra \bK$, $\l \mapsto \l^p$ is called  the Frobenius monomorphism and   $\overline{\mF}_p$ is  the algebraic closure of the field $\Fp = \Z / p\Z$, 

\item $L/K$ is a finite field extension,
 
\item  $G:=G(L/K)$ is the automorphism group of the $K$-algebra $L$. If $L/K$ is a Galois field extension then  $G(L/K)$ is called the Galois group of $L/K$, 

\item $L^{G(L/K)}:=L^G:=\{ l\in L\, | \, \s (l)=l$ for all $\s\in G(L/K)\}$ is the field of $G(L/K)$-invariants,

\item  $\CF (L/K)$ and  $\CG (L/K)$ are the sets of all and Galois subfields of $L/K$, respectively,

\item    $\CG (G(L/K))$ and  $\CN(G(L/K))$ are  the sets of all and normal subgroups of $G(L/K)$, respectively, 

\item $\O_{L/K}$ is the module of K\"{a}hler differentials of the $K$-algebra $L$,

\item $\CD :=\CD (L/K)=L\oplus \CD (L/K)_+$ is the algebra of $K$-linear differential operators on the field extension $L/K$ where 
$$\CD (L/K)_+:=\{\d \in \CD (L/K)\, | \, \d (1)=0\}$$
 is the left ideal of differential operators with  zero  constant term,

\item $L^{ \CD (L/K)_+}:=\{ l\in L\, | \, \d(l)=0$ for all $\d \in  \CD (L/K)_+\}$ is the field of $ \CD (L/K)_+$-constants,

\item $\Der_K(L):=\Der (L/K)$ is the Lie algebra of $K$-derivations,


\item The field $L_{dif}:=(L/K)_{dif} :=
\{ l\in L\, | \, \d l = l\d$ for all $\d \in \CD (L/K)\}$  is the centralizer in $L$ of the algebra of differential operators on $L/K$,


\item $E:=E(L/K):=\End (L/K):=\End_K (L)\simeq M_n(K)$ is the endomorphism algebra of the $K$-vector space $L$ and  $M_n(K)$ is the algebra of $n\times n$ matrices over $K$ where $n=[L:K]:=\dim_K(L)$ is the degree of the field extension $L/K$,

\item For  subsets $S$ and $T$ of $E(L/K)$, $C_{E(L/K)}(S):=\{ c\in E(L/K)\, | \, cs=sc$ for all $s\in S\}$ is the centralizer of $S$ in $E(L/K)$ and  $C_{T}(S):=\{ t\in T\, | \, ts=st$ for all $s\in S\}$ is the centralizer of $S$ in $T$, 

\item Fora  field $ M\in \CF (L/K)$, let
\begin{eqnarray*}
\CD (L/K)^M &:=& \{ \d \in \CD (L/K)\, | \, \d m=m\d\;\; {\rm for\; all}\;\; m\in M\}=C_{\CD (L/K)}(M),\\
 G(L/K)^M&:=& \{ g\in G(L/K)\, | \, g(m)=m\;\; {\rm for\; all}\;\; m\in M \}=C_{G(L/K)}(M),
\end{eqnarray*}

\item $L^{pi}$, $L^{sep}$, $L^{nor}$  and  $L^{gal}$ are the largest pure inseparable/separable/normal/ Galois subfields of the field extension $L/K$, respectively, 



\item $\CA (E(L/K),L)$ is the set of all $K$-subalgebras of $E(L/K)$ that contain the field $L$,

\item For an algebra $A\in  \CA (E(L/K),L)$,  let
\begin{eqnarray*}
L^{A\cap \CD (L/K)_+}  &:=& \{ l \in L\, | \, \d (l)=0\;\; {\rm for\; all}\;\; \d\in A\cap \CD (L/K)_+\},\\
&\stackrel{{\rm Cor.}\, \ref{a1Jun25}}{=}&\{ l \in L\, | \, \d l=l\d\;\; {\rm for\; all}\;\; \d\in A\cap \CD (L/K)_+\},\\
 L^{A\cap G(L/K)}&:=& \{  l \in L\, | \, g(l)=l\;\; {\rm for\; all}\;\; g\in A\cap G(L/K) \},
\end{eqnarray*} 

\item $\CA (E(L/K),L, G(L/K)):=\{ A\in  \CA (E(L/K),L)\, | \, gAg^{-1}\subseteq A$ for all $g\in G(L/K)\}$,

\item For a field $F\in \CF (L/K)$, let 
$\CA (E(L/K),L, F) := \{A\in  \CA (E(L/K),L)\, | \, C_{E(L/K)}(F)\subseteq A\}$,

\item $\CA (E(L/K),L, F, G(L/K)):=\{ A\in \CA (E(L/K),L, F)\, | \, g A g^{-1}= A\;\; {\rm for \; all}\;\; g\in G(L/K)\}$,



\item $\CL (L/K)$ is the  set of balanced Lie subalgebras $\CG$ of the Lie algebra $\CD (L/K)_+$, i.e.  $C_L(\CG) = L^\CG$,

\item $\mL (L,K)$ is the set of dominant Lie algebras, 

\item $\mL G(L/K):=\{ (\CG, H)\in \mL (L/K)\times \CG (G(L/K))\, | \, h\CG h^{-1}=\CG$    for all elements $h\in H\}$ is the set of dom-groups of $L/K$,


\item $\fCK$ is  the class of   central simple finite dimensional  $K$-algebras and  $\fCdK$ is  the class of  central simple finite dimensional  division $K$-algebras,


\end{itemize}

   The following two results are the main results of the classical Galois Theory: {\em Let a finite field extension $L/K$ be a Galois field extension. Then:}
  \begin{enumerate}

\item {\sc (The Galois Correspondence for subfields)} 
{\em The map
$$
\CF (L/K)\ra \CG (G(L/K)), \;\; M\mapsto  G(L/M)
$$ 
is a bijection with inverse} $
H\mapsto L^H.$

\item {\sc (The  Galois Correspondence for Galois subfields)} 
{\em The map
$$
\CG (L/K)\ra \CN (G(L/K)), \;\; \G\mapsto  G(L/\G)
$$ 
is a bijection with inverse} $
\G\mapsto L^{\G}.$

\end{enumerate}

Finite purely inseparable field extensions have trivial automorphism group but rich subfield structure. They are normal field extensions but normal field extensions are a larger class. For a long time it was an open problem to produce a `Galois Theory' for normal finite field extensions and the class of finite purely inseparable field extensions was a natural starting point. 

  In \cite{Jacobson-1937, Jacobson-1944} (1937, 1944), Jacobson  introduced an analogue  of the Galois theory for {\em purely inseparable field  extensions  $L/K$ of exponent 1} (meaning that the $p$'th power of every element of $L$ is in $K$), where the Galois groups (which are trivial) are replaced by {\em restricted Lie algebras of derivations}.  
 Jacobson's theory establishes a correspondence between intermediate fields of a purely inseparable extension of exponent 1 and restricted Lie subalgebras of derivations.   The relationship between derivations and purely inseparable field extensions of exponent 1 was studied by Baer \cite{Baer-1927} (1927).

 A field extension is called {\em simple} or {\em primitive} if it  is generated by a single element.  A purely inseparable extension is called a {\em modular extension} if it is a tensor product of finitely  many  simple field extensions. Modular extensions can be seen as `free' objects in the class of finite purely inseparable field extensions (in the same vein as polynomial algebras are `free' objects in the Algebraic Geometry), i.e. they satisfy no relations (apart from the obvious ones). In particular, {\em every extension of exponent 1 is modular}, but there are non-modular extensions of exponent 2 as shown by Weisfeld \cite{Weisfeld-1965} (1965). In more detail, let us   consider subfields  $K=\Fp (x^p,y^p,z^{p^2})$ and $L=K (z,xz+y)$ of the field $\Fp (x,y,z)$ of rational functions in 3 variables.  Weisfeld \cite{Weisfeld-1965}  proved that the field extension $L/K$ is purely inseparable field extension of exponent 2 which is not modular.

  Sweedler \cite{Sweedler-1968} (1968) and Gerstenhaber  and  Zaromp \cite{Gerstenhaber-Zaromp-1970} (1970) produced an analogue of the  Galois Theory for  {\em modular} extensions, where restricted Lie algebras of  derivations are replaced by {\em higher derivations}, i.e. the {\em Hasse-Schmidt derivations}, \cite{Hasse-Schmidt-1937} (1937). A survey of Galois theory for inseparable field extensions is given in the paper of  Deveney  and Mordeson \cite{Deveney-Mordeson-1996} (1996).

 In \cite{AnGaloisTh-NORMAL-Fields}, Bavula  introduced B-extensions which are the {\em most symmetrical} finite field extensions (a finite field extension $L/K$ is called a {\em B-extension}  if the endomorphism algebra $\End_K(L)$ is generated by the algebra of differential operators $\CD (L/K)$ on the $K$-algebra  $L$ and the automorphism group $G(L/K):=\Aut_{K-{\rm alg}}(L)$) and  obtained an analogue of the Galois Theory for B-extensions.  Surprisingly, it turned out the class of B-extensions coincided with the class of {\em normal} finite field extensions  giving a new characteristic property of them. As a result, an analogue of the Galois Theory is obtained for normal field extensions \cite{AnGaloisTh-NORMAL-Fields}. In particular, all Galois field extensions and all purely inseparable field extensions are B-extensions. Bavula's approach is a ring theoretic (characteristic free)  approach which is based on central simple algebras.
In this approach, analogues of the Galois Correspondences (for subfields and normal  subfields of $L$) are deduced from the Double Centralizer Theorem which is applied to the central simple algebra  $\End_K(L)$ and subfields of B-extensions. 

  Since Galois finite field extensions are B-extensions, this approach gives a new conceptual (short) proofs of the  key results of the Galois Theory, see \cite{GaloisTh-RingThAp}  for details. It also reveals that the `{\bf maximal symmetry}' (of {\em normal} field extensions) is the essence of the classical Galois Theory and the analogue of the Galois Theory for normal field extensions.

 For a finite field extension $L/K$,  notice that $\CD (L/K)=L\oplus \CD (L/K)_+$ where 
$$\CD (L/K)_+:=\{\d \in \CD (L/K)\, | \, \d (1)=0\}$$
 is the {\bf left ideal of proper  differential operators} on $L/K$  which is a {\bf Lie algebra} with commutator of elements as the Lie bracket. The left $\CD (L/K)$-module $L$ is isomorphic to the 
 factor module $\CD (L/K)/\CD (L/K)_+$.
  For a subset $S\subseteq \CD (L/K)$, let 
  $$
  C_L(S) := \{l\in L\, |\, ls=sl\; {\rm  for\; all}\;s\in S \}\;\; {\rm and}\;\; 
  L^S := \{l\in L\, |\, s(l)=0\; {\rm  for\; all }\; s\in S\}.
  $$
   Clearly, $C_L(S) \in \CF (L/K)$.  In  \cite{AnGaloisTh-NORMAL-Fields}, it is proved that $C_L(S)\subseteq L^S$.

\begin{definition} 
(\cite{AnGaloisTh-NORMAL-Fields}) A Lie subalgebra $\CG$ of the Lie algebra $\CD (L/K)_+$ is called {\bf balanced} if  
$C_L(\CG) = L^\CG$. The set of balanced Lie subalgebras  of $\CD (L/K)_+$ is denoted by $\CL (L/K)$.  
A Lie algebra $\CG \in \CL (L/K)$ is called a {\bf dominant Lie algebra} if $\CG'\subseteq \CG$ 
for all Lie algebras $\CG' \in \CL (L/K)$  such that $L^{\CG'}=L^\CG$. Let $\mL (L/K)$ be the set  of dominant Lie algebras and 
$$
\mL G(L/K):=\{ (\CG, H)\in \mL (L/K)\times \CG (G(L/K))\, | \, h\CG h^{-1}=\CG\;\; {\rm  for\; all\; elements}\;\; h\in H\}.
$$  An element $(\CG, H)\in \mL G(L/K)$ is called a {\bf dom-group} or a {\bf dominant pair} and the set $\mL G(L/K)$ is called the {\bf dom-group set} or the 
{\bf set of dominant pairs} of the field extension $L/K$.
\end{definition}

 A dominant Lie algebra and a dom-group are two of the key concepts  in the analogue of the Galois correspondence for {\em normal} finite  field extensions (\cite{AnGaloisTh-NORMAL-Fields}, Theorem \ref{1Jun25}). The concept of  a dominant Lie algebra is one  of the key concepts  in the analogue of the Galois correspondence for {\em purely inseparable}   finite  field extensions (\cite{AnGaloisTh-NORMAL-Fields}, Theorem \ref{D24Mar25}).

In the Galois theory,  the two Galois Correspondences are given in terms of the fields $L^H$ of $H$-{\em invariants} where $H$ is a subgroup of the Galois group $G(L/K)$. But for normal field extensions,  analogues of the two Galois Correspondences are given in terms of the fields  $L^H\cap L^{\CG}$ of $H$-invariants and $\CG$-{\em invariants/constants} where $\CG$ is a dominant Lie algebra, \cite{AnGaloisTh-NORMAL-Fields} (Theorem  \ref{1Jun25} and  Theorem  \ref{11Jun25}). 
   

As a particular case, we obtain  a  Galois Theory for purely inseparable finite  field extensions $L/K$ where the algebra  of differential operators $\CD(L/K)$ or the dominant Lie algebra  $\CD(L/K)_+$ plays the role of the Galois group $G(L/K)$ in the Galois case (which is the identity group in this case) and $G(L/K)$-invariants are replaced by $\CD (L/K)_+$-invariants/constants.  In  \cite{AnGaloisTh-NORMAL-Fields}, explicit descriptions of the subalgebras $\CD (L/K)$, $L\rtimes G(L/K)$ and $\CD (L/K)\rtimes G(L/K)$ of $E(L/K)$ are given and their properties are studied. These are the technical core of the paper \cite{AnGaloisTh-NORMAL-Fields} on which the analogue of the Galois Theory for normal   fields builds on.\\

{\bf Analogue of the  Galois Theory for   finite field extensions.}  Notice that for a  normal finite field extension $L/K=L^{pi}\t L^{gal}$  and a field $M\in \CF (L/K)$, the following  maps are  isomorphisms of Galois groups:
\begin{eqnarray*}
G(L/K)&\ra & G(L/L^{pi}), \;\; \;\; g\mapsto g,\\
 G(L/L^{pi})&\ra & G(L^{gal}/K), \;\;g\mapsto g|_{L^{gal}}: L^{gal}\ra L^{gal},\\
  G(L/ML^{pi})&\ra & G(L/M), \;\;\;\;\;\; g\mapsto g.
\end{eqnarray*}
Therefore, we may identify the Galois groups above: $ G(L/K)=G(L/L^{pi})=G(L^{gal}/K)$ and  $G(L/M)=G(L/ML^{pi})$.

\begin{definition} 
Let $F/K$ be a finite field extensions and $L/K$ be its normal closure.  Let 
\begin{eqnarray*}
\CL (L/K,F) &:=& \{\CG \in \CL (L/K)\,  | \, \CD (L/F)_+\subseteq \CG \},\\
\mL (L/K,F) &:=&  \{\CG \in \mL (L/K)\,  | \, \CD (L/F)_+\subseteq \CG \}=\mL (L/K) \cap \CL (L/K,F),\\
\CG (G(L/K),F) &:=& \{H\in \CG (G(L/K))\, | \, G(L/F)\subseteq H\},\\
\mL G(L/K,F) &:=& \{ (\CG, H)\in \mL (L/K, F)\times \CG (G(L/K), F)\, | \, h\CG h^{-1}=\CG\;\; {\rm  for\; all\; elements}\;\; h\in H\}.
\end{eqnarray*}
The Lie algebras  of $\CL (L/K,F)$ and  $\mL (L/K,F)$ are called  {\bf $F$-balanced} and {\bf  $F$-dominant}, respectively. The pairs $(\CG , H)\in \mL G(L/K,F)$ are called   {\bf $F$-dom-groups} or {\bf $F$-dominant pairs}. The sets $\CL (L/K,F)$, $\mL (L/K,F)$ and $\mL G(L/K,F)$ 
are called the sets of  {\bf $F$-balanced, $F$-dominant} Lie algebras   and  {\bf $F$-dom-groups/$F$-dominant pairs}, respectively.
\end{definition}
These concepts play a central role in the analogue of the Galois Theory for arbitrary finite field extensions.

Let $F/K$ be a  finite field extension of arbitrary characteristic  and $L/K$ be its normal closure. Notice that $C_{E(L/K)}(F)=E(L/F)$ and $FL^{gal}/L^{gal}$ is a purely inseparable finite field extension. Let 
\begin{eqnarray*}
  \CA (E(L/K),L, F) &:=& \{A\in  \CA (E(L/K),L)\, | \, E(L/F)\subseteq A\}, \\
\mL (L/L^{gal}, F) &:=& \{\CG\in \mL (L/L^{gal})\, | \,  \CD (L/FL^{gal})_+\subseteq \CG\}\stackrel{\cite{AnGaloisTh-NORMAL-Fields}}{=}\{ \CD (L/N)_+\, | \, N\in \CF (FL^{gal}/L^{gal})\},  \\
G(L/L^{pi}, F) &:=& \{ H\in \CG (G(L/L^{pi}))\, | \,  G(L/FL^{pi})\subseteq H\}. 
\end{eqnarray*}

 \begin{theorem}\label{Arb-1Jun25}
Let $F/K$ be a  finite field extension  and $L/K$ be its normal closure. Then:

\begin{enumerate}

\item 
\begin{eqnarray*}
\CA (E(L/K),L, F)&=&\Big\{ C_{E(L/K)}(M)
=\CD (L/M)\rtimes G(L/M)\\
&=&\CD (L/K)^M\rtimes G(L/K)^M\, | \, M\in \CF (L/K)\Big\}. 
\end{eqnarray*}
\item {\sc (Analogue of the  Galois correspondence for subfields of a  field extension)}

 The map
\begin{eqnarray*}
\CF (F/K)&\ra&  \CA (E(L/K), L, F), \;\; M\mapsto  C_{E(L/K)}(M)=\CD (L/M)\rtimes G(L/M)\\
&=&\CD (L/K)^M\rtimes G(L/K)^M
\end{eqnarray*}
is a bijection with inverse

\begin{eqnarray*}
A\mapsto C_{E(L/K)}(A)&=&C_{L}( A\cap \CD (L/K)_+)\cap L^{A\cap G(L/K)}=\Big(C_{L}( A\cap \CD (L/K)_+)) \Big)^{A\cap G(L/K)}\\
&=& L^{A\cap \CD (L/K)_+}\cap L^{A\cap G(L/K)}=\Big(L^{A\cap \CD (L/K)_+} \Big)^{A\cap G(L/K)}.\\
 \bigg( A=\CD (L/M)&\rtimes &G(L/M)\mapsto 
 L^{\CD (L/M)_+}\cap L^{G(L/M)}=\Big(L^{\CD (L/M)_+} \Big)^{G(L/M}.\bigg)
\end{eqnarray*}

\item $\mL G(L/K, F)=\Big\{ \Big(\CD (L/M)_+, G(L/M)\Big) \, | \, M\in \CF (F/K)\Big\}$ and for all fields $M\in \CF (F/K)$, 
$$
\Big(\CD (L/M)_+, G(L/M) \Big)=\Big(\CD (L/ML^{gal})_+, G(L/ML^{pi}) \Big)\in \mL (L/L^{gal})\times G(L/L^{pi}).
$$
 In particular, 
$$
\mL G(L/K,F)=\Big\{ (\CG , H)\in \mL (L/L^{gal}, F)\times \CG (G(L/L^{pi}),F)\, | \, h\CG h^{-1}=\CG\;\; {\rm  for\; all\; elements}\;\; h\in H\}.
$$

\item  {\sc (Analogue of the  Galois correspondence for subfields of a  field extension)}

 The map
$$
\CF (L/K)\ra \mL G(L/K,F),  \;\; M\mapsto  \Big(\CD (L/M)_+, G(L/M) \Big)=\Big(\CD (L/ML^{gal})_+, G(L/ML^{pi}) \Big)
$$ 
 is a bijection with inverse
$ (\CG , H)\mapsto L^\CG\cap L^H=\Big(L^\CG \Big)^H$ and $h(L^\CG)=L^\CG$ for all $h\in H$.

\item   For all fields $M\in \CF (F/K)$, $[M:K]|G(L/M)|\dim_K(\CD (L/M))=[L:K]^2$ or, equivalently, $|G(L/M)|\dim_M(\CD (L/M))=[L:M]^2$  or, equivalently, $|G(L/M)|\dim_M(\CD (L/M))=[L:M]^2$. 
\end{enumerate}

\end{theorem}
 
  The proof of Theorem  \ref{Arb-1Jun25} is given in Section \ref{GEN-CASE}. \\

{\bf Normal subfields of   finite field extensions  (two approaches).} Let $F/K$ be a  finite field extension,  $L/K$ be its normal closure in $\bK$. Then $L=L^{pi}\t L^{gal}$ and $L^{gal}=L^{sep}$. Let $F^{nor}=(F/K)^{nor}$ be the largest normal field extension of $K$ that is contained in the field $F$ (it is the sum of all normal field extensions in $F$). Then $F^{nor}=F^{pi}\t F^{gal}\subseteq L$ where $F^{pi}=F\cap L^{pi}$ and $ F^{gal}=F\cap L^{gal}$.
 There are two cases: a finite field  extension $F/K$ is inseparable (Theorem \ref{Arb-11Jun25} and Theorem \ref{NEWArb-11Jun25}) or separable ( Theorem \ref{SepNEWArb-11Jun25} and Theorem \ref{GalArb-11Jun25}).  Two approaches: In each of the two cases, the corresponding pair of theorems reflects the fact that, for the finite field extension $F/K$, the analogue of Galois theory is expressed in terms of the field extensions $L/K$ and $F^{nor}$ respectively.  From the point of view of applications, it is advantageous to have both types of results, because for the field extension $F/K$ either  $L/K$ or   $F^{nor}$ may be more convenient to work with.
 \begin{definition} 
Let $F/K$ be a finite field extensions and $L/K$ be its normal closure.  Let 
\begin{eqnarray*}
\CL (L/K,F,G(L/K)) &:=& \{\CG \in \CL (L/K,F)\,  | \;  g\CG g^{-1}\;\; {\rm for\; all}\;\;g\in G(L/K) \},\\
\mL (L/K,F,G(L/K)) &:=&  \{\CG \in \mL (L/K,F)\,  | \;  g\CG g^{-1}\;\; {\rm for\; all}\;\;g\in G(L/K) \}\\
&=&\mL (L/K, F) \cap \CL (L/K, F, G(L/K)),\\
\CN (G(L/K),F) &:=& \{H\in \CN (G(L/K))\, | \, G(L/F)\subseteq H\},\\
\mL G(L/K,F, G(L/K)) &:=& \{ (\CG, H)\in \mL (L/K, F, G(L/K))\times \CN (G(L/K), F)\}\\
&=& \mL (L/K, F, G(L/K))\times \CN (G(L/K), F),\\
\CA (E(L/K),L, F, G(L/K))&:=&\{ A\in \CA (E(L/K),L, F)\, | \, g A g^{-1}= A\;\; {\rm for \; all}\;\; g\in G(L/K)\}.
\end{eqnarray*}
The  Lie algebras  of  $\CL (L/K,F,G(L/K))$ and  $\mL (L/K,F,G(L/K))$ are called  {\bf normal $F$-balanced} and {\bf normal $F$-dominant}, respectively. The pairs $(\CG, H)\in \mL G(L/K,F,G(L/K))$ 
are called   {\bf normal $F$-dom-groups} or {\bf normal $F$-dominant pairs}, respectively. The sets $\CL (L/K,F,G(L/K))$, $\mL (L/K,F,G(L/K))$ and $\mL G(L/K,F,G(L/K))$ 
are called the sets of  {\bf normal $F$-balanced, normal $F$-dominant} Lie algebras   and  {\bf normal  $F$-dom-groups/normal $F$-dominant pairs}, respectively.  
\end{definition}
These concepts are central in the  analogue of the Galois Correspondence  for {\em normal} subfields of a finite field extensions (Theorem \ref{Arb-11Jun25}). 

 Theorem  \ref{Arb-11Jun25}.(1) describes normal subfields of the finite field extension $F/K$.
 Theorem  \ref{Arb-11Jun25}.(2) gives an explicit description of the set of algebras  $\CA (E(L/K),L, F,G(L/K))$. Theorem  \ref{Arb-11Jun25}.(3) establishes an order reversing bijection between 
 normal subfields of  $F/K$ and the set of algebras $\CA (E(L/K),L, F,G(L/K))$. 
Theorem  \ref{Arb-11Jun25}.(4)  gives an explicit description of the  elements of the set $\mL G(L/K, F, G(L/K))$. 
 Theorem  \ref{Arb-11Jun25}.(5) establishes an order reversing bijection between 
 normal subfields of  $F/K$ and the set of algebras $\mL G(L/K, F, G(L/K))$. 

For the finite field extension $F/K$, let $\CN (F/K)$ be the set of normal subfields $N/K$ of $F/K$ and  $\CN(G(L/K))$ be the set of normal subgroups of the group $ G(L/K)$. 

\begin{theorem}\label{Arb-11Jun25}
Let $F/K$ be a  finite field extension
 and $L/K$ be its normal closure. Then:

\begin{enumerate}

\item  For each field $M\in  \CN (F/K)$, $M=M^{pi}\t M^{gal}$  where $M^{pi}=M\cap F^{pi}
=M\cap L^{pi}
$ and $M^{gal}=M\cap F^{gal}=M\cap L^{gal}$, and 
$$\CN (F/K)=\CN (F^{nor}/K)=\CF \Big(F^{pi}/K\Big)\t \CN  \Big(F^{gal}/K\Big)
:=\Big\{ N\t \G\, | \, N\in \CF  \Big(F^{pi}/K\Big),\,  \G\in \CN  \Big(F^{gal}/K\Big)\Big\}.
$$

\item $\CA (E(L/K),L, F, G(L/K))=\Big\{\CD \Big(L^{pi}/N\Big)\t \Big(L^{gal}\rtimes G(L^{gal}/\G) \Big)= E\Big(L^{pi}/N\Big)\t  E\Big(L^{gal}/\G\Big)\, | \,$ $ N\in \CF \Big(F^{pi}/K\Big), \,  \G\in \CN  (F^{gal}/K)  \Big\}$.

\item {\sc (Analogue of the  Galois correspondence for normal  subfields of a field extension)}

The map
\begin{eqnarray*}
\CN (F/K)&\ra & \CA (E(L/K),L, F, G(L/K)), \;\; M\mapsto  C_{E(L/K)}(M)=C_{E(L^{pi}/K)}(M^{pi})\t C_{E(L^{gal}/K)}(M^{gal})\\
&=&E(L^{pi}/M^{pi})\t E(L^{gal}/M^{gal})=\CD (L^{pi}/M^{pi})\t \Big(L^{gal}\rtimes G(L^{gal}/M^{gal}) \Big)
\end{eqnarray*}
is a bijection with inverse
\begin{eqnarray*}
 A&\mapsto& C_{E(L/K)}(A)=\Big( L^{pi}\Big)^{A\cap \CD (L^{pi}/K)_+}\t \Big( L^{gal}\Big)^{A\cap G(L^{gal}/K)}.\\
 \bigg(A&=&\CD (L^{pi}/N)\t \Big(L^{gal}\rtimes G(L^{gal}/\G) \Big) \mapsto   \Big( L^{pi}\Big)^{ \CD (L^{pi}/N)_+}\t \Big( L^{gal}\Big)^{ G(L^{gal}/\G)}.\bigg)
 \end{eqnarray*}

 \item $\mL G(L/K, F, G(L/K))=\Big\{ \Big(\CD (L/M)_+, G(L/M)\Big) \, | \, M\in \CN (F/K)\Big\}$ and for all normal  fields $M\in \CN (F/K)$, 
$$
\Big(\CD (L/M)_+, G(L/M) \Big)=\Big(\CD (L/ML^{gal})_+, G(L/ML^{pi}) \Big)\in \mL (L/L^{gal})\times \CN (G(L/L^{pi})).
$$
In particular, $\mL G(L/K,F, G(L/K))=\mL (L/L^{gal}, F,G(L/K))\times \CN (G(L/L^{pi}),F)$.

\item  {\sc (Analogue of the  Galois correspondence for normal subfields of a  field extension)}

 The map
 \begin{eqnarray*}
\CN (L/K)&\ra& \mL G(L/K,F, G(L/K)) =\mL (L/K, F, G(L/K))\times \CN (G(L/K), F),\\
 M & \mapsto & \Big(\CD (L/M)_+, G(L/M) \Big)=\Big(\CD (L/ML^{gal})_+, G(L/ML^{pi}) \Big)
\end{eqnarray*}
 is a bijection with inverse
$ (\CG , H)\mapsto L^\CG\cap L^H$ and $h(L^\CG)=L^\CG$ for all $h\in H$.
 
\end{enumerate}
\end{theorem}
The proof of Theorem \ref{Arb-11Jun25} is given in Section \ref{GEN-CASE}.

Suppose that $F/K$ is  a  separable finite field extension. Notice that in zero characteristic all field extensions are separable. For separable field extensions the concepts `normal' and `Galois' coincide. Let $\CG (F/K)$ be the set of Galois=normal  intermediate subfields in $F/K$. Theorem \ref{GalArb-11Jun25}.(2) 
is the  Galois correspondence for Galois=normal  subfields of the  field extension $F/K$.

\begin{theorem}\label{GalArb-11Jun25}
 Let $F/K$ be a  separable finite field extension of characteristic $p$ and $L/K$ be its normal=Galois closure. Then:

\begin{enumerate}

\item $\CA (E(L/K),L, F, G(L/K))=\{L\rtimes G(L/\G) \, | \,$ $ \G\in \CG  (F/K)  \}$.

\item {\sc (The  Galois correspondence for Galois  subfields of a separable field extension)} 

The map
\begin{eqnarray*}
\CG (F/K)&\ra & \CA (E(L/K),L, F, G(L/K)), \;\; M\mapsto  C_{E(L/K)}(M)= E(L/M)\\
&=&L\rtimes G(L/M) 
\end{eqnarray*}
is a bijection with inverse
\begin{eqnarray*}
 A&\mapsto& C_{E(L/K)}(A)=  L^{A\cap G(L/K)}.\\
 \bigg(A&=&L\rtimes G(L/\G)  \mapsto   L^{ G(L/\G)}.\bigg)
 \end{eqnarray*}
 
 \item $\mL G(L/K, F, G(L/K))=\Big\{ G(L/M) \, | \, M\in \CG (F/K)\Big\}=\CN (G (L/K), F)$.

\item  {\sc (Analogue of the  Galois correspondence for Galois subfields of a  field extension)}

 The map
 $$
 \CG (F/K)\ra \CN (G (L/K), F),\;\; 
 M  \mapsto G(L/M)
 $$
 is a bijection with inverse
$ H\mapsto  L^H$. 
 
\end{enumerate}
\end{theorem}
The proof of Theorem \ref{GalArb-11Jun25} is given in Section \ref{SEPARABLE}.

 Theorem  \ref{NEWArb-11Jun25}.(1) gives an explicit description of the set of algebras   $\CA (E(F^{nor}/K),F^{nor},  G(F^{nor}/K))$. Theorem  \ref{NEWArb-11Jun25}.(2) establishes an order reversing bijection between 
 normal subfields of  $F/K$ and the set of algebras $\CA (E(F^{nor}/K),F^{nor},  G(F^{nor}/K))$.  Theorem  \ref{NEWArb-11Jun25}.(3) establishes an order reversing bijection between 
 normal subfields of  $F/K$ and the elements of the set  $\mL (F^{pi}/K)\times \CN (G(F^{gal}/K))$.

 \begin{theorem}\label{NEWArb-11Jun25}
Let $F/K$ be a  finite field extension of characteristic $p$. Then:

\begin{enumerate}


\item $\CA (E(F^{nor}/K),F^{nor},  G(F^{nor}/K))=\Big\{\CD \Big(F^{pi}/N\Big)\t \Big(F^{gal}\rtimes G(F^{gal}/\G ) \Big)= E\Big(F^{pi}/N\Big)\t E\Big(F^{gal}/\G \Big)\, | \,$ $ N\in \CF \Big(F^{pi}/K\Big),\,  \G\in \CN  (F^{gal}/K)  \Big\}$.


\item {\sc (Analogue of the  Galois correspondence for normal  subfields of an inseparable  finite field extension)} 

The map
\begin{eqnarray*}
\CN (F/K)&\ra & \CA (E(F^{nor}/K),F^{nor},  G(F^{nor}/K)), \\
 M &\mapsto & C_{E(F^{nor}/K)}(M)=C_{E(F^{pi}/K)}(M^{pi})\t C_{E(F^{gal}/K)}(M^{gal})\\
&=&E(F^{pi}/M^{pi})\t E(F^{gal}/M^{gal})\\
&=&\CD (F^{pi}/M^{pi})\t \Big(F^{gal}\rtimes G(F^{gal}/M^{gal}) \Big)
\end{eqnarray*}
is a bijection with inverse
\begin{eqnarray*}
 A&\mapsto& C_{E(F^{nor}/K)}(A)=\Big( F^{pi}\Big)^{A\cap \CD (F^{pi}/K)_+}\t \Big( F^{gal}\Big)^{A\cap G(F^{gal}/K)}.\\
 \bigg(A&=&\CD (F^{pi}/N)\t \Big(F^{gal}\rtimes G(F^{gal}/\G) \Big) \mapsto   \Big( F^{pi}\Big)^{ \CD (F^{pi}/N)_+}\t \Big( F^{gal}\Big)^{ G(F^{gal}/\G)}.\bigg)
\end{eqnarray*}


\item {\sc (Analogue of the  Galois correspondence for normal  subfields of an inseparable  finite field extension)} 

The map
$$
\CN (F/K)\ra \mL (F^{pi}/K)\times \CN (G(F^{gal}/K)), \;\; M=M^{pi}\t M^{gal}\mapsto \Big(\CD (F^{pi}/M^{pi})_+, G(F^{gal}/M^{gal})\Big)
$$
 is a bijection with inverse $(\CG, H)\mapsto \Big(F^{pi} \Big)^\CG\t \Big( F^{gal}\Big)^H$.

\end{enumerate}
\end{theorem}
 
The proof of Theorem \ref{NEWArb-11Jun25}   is given in Section \ref{GEN-CASE}.

 \begin{theorem}\label{SepNEWArb-11Jun25}
Let $F/K$ be a separable finite field extension. Then:

\begin{enumerate}

\item $F^{nor}=F^{gal}$ and $\CN (F/K)=\CN (F^{gal}/K)=\CG  (F^{gal}/K)$. 

\item $\CA (E(F^{gal}/K),F^{gal},  G(F^{gal}/K))=\Big\{ F^{gal}\rtimes G(F^{gal}/\G) =  E(F^{gal}/\G) \, | \,$ $ \G\in \CN  (F^{gal}/K)  \Big\}$.


\item {\sc (Analogue of the  Galois correspondence for normal  subfields of a separable  finite field extension)} 

The map
\begin{eqnarray*}
\CN (F/K)&\ra & \CA (E(F^{gal}/K),F^{gal},  G(F^{gal}/K)), \\
 M &\mapsto & C_{E(F^{gal}/K)}(M)= E(F^{gal}/M)=F^{gal}\rtimes G(F^{gal}/M)
\end{eqnarray*}
is a bijection with inverse
\begin{eqnarray*}
 A&\mapsto& C_{E(F^{gal}/K)}(A)= \Big( F^{gal}\Big)^{A\cap G(F^{gal}/K)}.\\
 \bigg(A&=&F^{gal}\rtimes G(F^{gal}/\G)  \mapsto   \Big( F^{gal}\Big)^{ G(F^{gal}/\G)}.\bigg)
\end{eqnarray*}


\item {\sc (Analogue of the  Galois correspondence for normal  subfields of a separable  finite field extension)} 

The map
$$
\CN (F/K)\ra \CN (G(F^{gal}/K)), \;\; M\mapsto G(F^{gal}/M)
$$
 is a bijection with inverse $H\mapsto \Big( F^{gal}\Big)^H$.

\end{enumerate}
\end{theorem}
 The proof of Theorem \ref{SepNEWArb-11Jun25}   is given in Section \ref{GEN-CASE}.

 The paper is organized as follows: At the beginning of Section \ref{GEN-CASE}, we recall some results of the papers  \cite{GaloisTh-RingThAp} and \cite{AnGaloisTh-NORMAL-Fields} that are used in the proofs of the paper. The proofs of Theorem  \ref{Arb-1Jun25},  Theorem  \ref{Arb-11Jun25} and Theorem \ref{NEWArb-11Jun25} are given.
  
  In Section \ref{SEPARABLE}, an analogue of the  Galois correspondence for subfields of a {\em separable} finite  field extension is given (Corollary \ref{Sep-Arb-1Jun25}.(2)). Theorem \ref{Gal-11Jun25}.(3) is the  Galois correspondence for Galois  subfields of a separable finite field extension.
  
 In Section  \ref{PURELY-INSEP}, an analogue of the Galois Theory for  {\em purely inseparable} finite field extensions is presented. In this section,  $K$ is a field of characteristic $p$ and $L/K$ is a purely inseparable   finite field extension. Theorem \ref{C24Mar25} provides an explicit description of the algebra of differential operators  $\CD (L/K)$. Theorem \ref{D24Mar25} is an analogue of the Galois Theory for purely inseparable field extensions that establishes a bijection between subfields and the algebras  of   differential operators.  
 All the results of this section are proved in \cite{AnGaloisTh-NORMAL-Fields} and they are included for the sake of completeness (as every subfield of a  purely inseparable finite field extension is also  a  purely inseparable finite  field extension).
 
  In Section \ref{EXAMPLE},  we consider  an example  of the Galois correspondences in Theorem \ref{Arb-1Jun25}.(2) and Theorem \ref{Arb-11Jun25}.(2), see Proposition \ref{XA7Jul25}.  It also provides an example of a normal finite field extension in prime characteristic such that  all subfields are normal. 
For a normal finite field extension $L/K$,  Proposition \ref{Ext-A12Apr25}   provides explicit descriptions of the subfields $L^{pi}$ and  $L^{gal}$ of $L/K$ and the automorphism group $G(L/K)$.


\section{Analogue of the Galois Theory for   finite field extensions} \label{GEN-CASE} 

 In this section, the proofs of Theorem  \ref{Arb-1Jun25},  Theorem  \ref{Arb-11Jun25} and Theorem \ref{NEWArb-11Jun25}  are given.  At the beginning of the section we recall some results of the papers  \cite{GaloisTh-RingThAp} and \cite{AnGaloisTh-NORMAL-Fields} that are used in the proofs. \\

{\bf B-extensions, D-extensions and G-extensions of fields.}  Given a $K$-algebra $A$, a group $\CG$ and a group homomorphism $\phi: \CG\ra \Aut_K(A)$, $g\mapsto \phi_g$, then a direct sum of $A$-modules,
$\CA := \bigoplus_{g\in \CG}Ag$,  has a  $K$-algebra structure that is given by the rule: For all elements $a,b\in A$ and $g,h\in \CG$, 
$$ ag\cdot bh=a\phi_g(b)gh.$$
The algebra $\CA$ is called a {\em skew group algebra}.

For a finite field extension $L/K$, the algebra $E(L/K)$ contains the group $G(L/K)$ and the algebra $\CD (L/K)$ of differential operators on the $K$-algebra $L$. A $K$-subalgebra of $E(L/K)$ which generated by the group $G(L/K)$ and the algebra $\CD (L/K)$ is the skew group algebra 
$$
\CD (L/K)\rtimes G(L/K)=\bigoplus_{g\in G(L/K)}\CD (L/K)g\subseteq E(L/K),\;\;  \cite[{\rm Theorem}\; 5.15]{AnGaloisTh-NORMAL-Fields}.
$$ 
 \begin{definition}
(\cite{AnGaloisTh-NORMAL-Fields}) A finite field extension $L/K$ is called a {\bf B-extension} if 
 $$E(L/K)=\CD (L/K)\rtimes G(L/K).$$
 \end{definition}
 The notation `B' abbreviates ‘Bi’ = ‘2’: both fundamental structures associated with the field extension $L/K$ - the algebra of differential operators $\CD (L/K)$ and the automorphism group $G(L/K)$ - are essential to the definition. 
 By the definition,  a field  extension $L/K$ is a B-extension iff the algebra $\CD (L/K)\rtimes G(L/K)$ is as large as possible. So, B-extensions are the most symmetrical field extensions. 
 \begin{definition}
 (\cite{AnGaloisTh-NORMAL-Fields}) A finite field extension $L/K$ is called a {\bf G-extension} if 
 $$E(L/K)=L\rtimes G(L/K).$$
 A finite field extension $L/K$ is called a {\bf D-extension} if 
 $$E(L/K)=\CD(L/K).$$
 \end{definition}
Above, `G' and `D' stand for `the Galois group' and `the algebra of differential operators', respectively. By the definition, G- and D-extensions are B-extensions. There are many B-extensions that are neither G- nor D-extensions. They are precisely normal field extensions that are neither Galois nor purely inseparable field extensions, \cite[{\rm Corollary}\;   6.7]{AnGaloisTh-NORMAL-Fields}.

\begin{proposition}\label{A8May25}
(\cite{AnGaloisTh-NORMAL-Fields})  A finite field extension  is a G-extension iff it  is a Galois field extension.
\end{proposition}

\begin{corollary}\label{VVB-a4May25}
 (\cite{AnGaloisTh-NORMAL-Fields}) A finite field extension is a D-extension iff it is a purely inseparable field extension.
\end{corollary}

\begin{proposition}\label{VB-aC24Mar25}
(\cite{AnGaloisTh-NORMAL-Fields})  A finite field extension $L/K$  is separable iff $\CD (L/K)=L$.
\end{proposition}

 So, intuitively, purely inseparable field extensions in prime characteristic are direct analogue of Galois field extensions  in the separable situation where the algebras  $E(L/K)=\CD(L/K)$ and $E(L/K)=L\rtimes G(L/K)$ are `swapped', i.e. the skew group algebra $L\rtimes G(L/K)$ is replaced by the algebras of differential operators $\CD(L/K)$ (automorphisms are replaced by differential operators). The general situation for B-extensions (in prime characteristic), is a mixture of two cases: the Galois and the purely inseparable ones.\\

 {\bf Analogue of the Galois Theory for normal field extensions.}  Theorem  \ref{1Jun25} is one of the main results of the paper. For normal  field extensions, it establishes an analogue of the Galois correspondence.

\begin{definition} 
For a finite field extension $L/K$ and its intermediate subfield $K\subseteq M\subseteq L$, let
\begin{eqnarray*}
\CD (L/K)^M &:=& \{ \d \in \CD (L/K)\, | \, \d m=m\d\;\; {\rm for\; all}\;\; m\in M\}=C_{\CD (L/K)}(M),\\
 G(L/K)^M&:=& \{ g\in G(L/K)\, | \, g(m)=m\;\; {\rm for\; all}\;\; m\in M \}=C_{G(L/K)}(M).
\end{eqnarray*}
\end{definition} 
 By the definitions, $\CD (L/K)^M$ is a subalgebra of the algebra $\CD (L/K)$ that contains $L$ and $G(L/K)^M$ is a subgroup of the group $G(L/K)$.

\begin{corollary}\label{a1Jun25}
(\cite{AnGaloisTh-NORMAL-Fields}) 
Let $L/K$ be a normal finite field extension of characteristic $p$. Then for an algebra $A\in  \CA (E(L/K),L)$,
$$
C_{L}( A\cap \CD (L/K)) = C_{L}( A\cap \CD (L/K)_+) = L^{A\cap \CD (L/K)_+}.
$$
\end{corollary}

\begin{definition}
For a finite field extension $L/K$ and an algebra $A\in  \CA (E(L/K),L)$,  let
\begin{eqnarray*}
C_{L}( A\cap \CD (L/K)_+) &:=& \{ l \in L\, | \, \d l=l\d\;\; {\rm for\; all}\;\; \d\in A\cap \CD (L/K)_+\}\\
&\stackrel{{\rm Cor.}\, \ref{a1Jun25} }{=}& L^{A\cap \CD (L/K)_+} :=
\{ l \in L\, | \, \d (l)=0\;\; {\rm for\; all}\;\; \d\in A\cap \CD (L/K)_+\},\\
 L^{A\cap G(L/K)}&:=& \{  l \in L\, | \, g(l)=l\;\; {\rm for\; all}\;\; g\in A\cap G(L/K)\}=C_{L}( A\cap G (L/K)).
\end{eqnarray*}
\end{definition} 
 By the definition, the sets $L^{A\cap \CD (L/K)_+}$ and $L^{A\cap G(L/K)}$ are subfields of $L$ that contain the field $K$ and the intersection $A\cap G(L/K)$ is a subgroup of $G(L/K)$.

\begin{lemma}\label{c1Jun25}
Let $L/K$ be a finite field extension and $M\in \CF (L/K)$. Then:
\begin{enumerate}

\item  $L\subseteq E(L/M)\subseteq E(L/K)$ and 
$\CD (L/K)^M=\CD (L/M) =E(L/M)\cap \CD (L/K)\subseteq  E(L/K)$.

\item If, in addition, the field extension $L/K$ is normal then  $G(L/M)=G(L/K)^M$.

\end{enumerate}

\end{lemma}

\begin{proof} 1.  It follows from the definition of the algebra of differential operators and the inclusions
$L\subseteq E(L/M)\subseteq E(L/K)$ that 
$\CD (L/M) =E(L/M)\cap \CD (L/K)\subseteq  E(L/K)$. Clearly, $E(L/M)\cap \CD (L/K)=\CD (L/K)^M$.

2. Statement 2 follows from the normality of the field extensions $L/K$ and $L/M$ and the inclusion $G(L/M)\subseteq G(L/K)$. 
\end{proof}

 For normal field extensions, Theorem \ref{1Jun25} gives an order reversing  Galois-type correspondence for their subfields.

 \begin{theorem}\label{1Jun25}
(\cite{AnGaloisTh-NORMAL-Fields}) Let $L/K$ be a normal finite field extension of characteristic $p$. Then:

\begin{enumerate}

\item $ \CA (E(L/K),L)=\{ C_{E(L/K)}(M)=\CD (L/M)\rtimes G(L/M)=\CD (L/K)^M\rtimes G(L/K)^M\, | \, M\in \CF (L/K)\}$. 

\item  {\sc (Analogue of the  Galois correspondence for subfields of a normal  field extension)}

The map
$$
\CF (L/K)\ra \CA (E(L/K),L), \;\; M\mapsto  C_{E(L/K)}(M)=\CD (L/M)\rtimes G(L/M)=\CD (L/K)^M\rtimes G(L/K)^M
$$ is a bijection with inverse
\begin{eqnarray*}
A\mapsto C_{E(L/K)}(A)&=&C_{L}( A\cap \CD (L/K)_+)\cap L^{A\cap G(L/K)}=\Big(C_{L}( A\cap \CD (L/K)_+)) \Big)^{A\cap G(L/K)}\\
&=& L^{A\cap \CD (L/K)_+}\cap L^{A\cap G(L/K)}=\Big(L^{A\cap \CD (L/K)_+} \Big)^{A\cap G(L/K)}.\\
 \bigg( A=\CD (L/M)&\rtimes &G(L/M)\mapsto 
 L^{\CD (L/M)_+}\cap L^{G(L/M)}=\Big(L^{\CD (L/M)_+} \Big)^{G(L/M}.\bigg)
\end{eqnarray*}

\item $\mL G(L/K)=\Big\{ \Big(\CD (L/M)_+, G(L/M)\Big) \, | \, M\in \CF (L/K)\Big\}$ and for all fields $M\in \CF (L/K)$, 
$$
\Big(\CD (L/M)_+, G(L/M) \Big)=\Big(\CD (L/ML^{gal})_+, G(L/ML^{pi}) \Big)\in \mL (L/L^{gal})\times G(L/L^{pi}).
$$
 In particular, 
$$
\mL G(L/K)=\Big\{ (\CG , H)\in \mL (L/L^{gal})\times \CG (G(L/L^{pi}))\, | \, h\CG h^{-1}=\CG\;\; {\rm  for\; all\; elements}\;\; h\in H\}.
$$

\item  {\sc (Analogue of the  Galois correspondence for subfields of a normal  field extension)}

 The map
$$
\CF (L/K)\ra \mL G(L/K),  \;\; M\mapsto  \Big(\CD (L/M)_+, G(L/M) \Big)=\Big(\CD (L/ML^{gal})_+, G(L/ML^{pi}) \Big)
$$ 
 is a bijection with inverse
$ (\CG , H)\mapsto L^\CG\cap L^H=\Big(L^\CG\Big)^H$ and $h(L^\CG)=L^\CG$ for all $h\in H$.

\item For all fields $M\in \CF (L/K)$, $[M:K]|G(L/M)|\dim_K(\CD (L/M))=[L:K]^2$  or, equivalently, $|G(L/M)|\dim_M(\CD (L/M))=[L:M]^2$.

\item The field extension $L/L^{gal}$ is a finite purely inseparable field extension and 
$\mL (L/L^{gal})=\{ \CD (L/N)_+\, | \, N\in \CF (L/L^{gal})\}$.

\end{enumerate}

\end{theorem}

Lemma \ref{d1Jun25} provides additional information on  the algebra $\CD (L/M)^K$ in Theorem \ref{1Jun25}.

\begin{lemma}\label{d1Jun25}
(\cite{AnGaloisTh-NORMAL-Fields})  Let $L/K$ be a normal finite field extension of characteristic $p$. Then $\CD (L/M)=\CD (L/K)^M=C_{E(L/ML^{gal})}(M)=E(L/L^{gal}M)=\CD (L/L^{gal}M)$.
\end{lemma}




For a finite  field extension $L/K$, let  $\CG (G)$ be the set of subgroups of the group  $G$. Theorem  \ref{Gal-1Jun25}.(3) is one of the main results of the classical Galois Theory.

 \begin{theorem}\label{Gal-1Jun25}
(\cite{GaloisTh-RingThAp})  Let $L/K$ be a G-extension (i.e. a Galois extension, by Proposition \ref{A8May25}). Then:

\begin{enumerate}

\item $ \CA (E,L)=\{ C_{E}(M)=L\rtimes G(L/M)=L\rtimes G^M\, | \, L\in \CF (L/K)\}$. 

\item The map
$$
\CF (L/K)\ra \CA (E,L), \;\; M\mapsto  C_{E}(M)=L\rtimes G(L/M)=L\rtimes G^M
$$ 
is a bijection with inverse
$$
A\mapsto C_{E}(A)=L^{A\cap G}.
$$
\item {\sc (The  Galois correspondence for subfields of a Galois  field extension)}

The map
$$
\CF (L/K)\ra \CG (G), \;\; M\mapsto  G(L/M)=G^M
$$ 
is a bijection with inverse $
H\mapsto L^H.$

\end{enumerate}
\end{theorem}

Theorem \ref{Gen-1Jun25} is an   analogue of the  Galois correspondence for subfields of a normal  field extension.

 \begin{theorem}\label{Gen-1Jun25}
Let $L/K$ be a normal finite field extension. Then:

\begin{enumerate}

\item $ \CA (E(L/K),L)=\{ C_{E(L/K)}(M)=\CD (L/M)\rtimes G(L/M)=\CD (L/K)^M\rtimes G(L/K)^M\, | \, M\in \CF (L/K)\}$. 

\item  {\sc (Analogue of the  Galois correspondence for subfields of a normal  field extension)}

The map
$$
\CF (L/K)\ra \CA (E(L/K),L), \;\; M\mapsto  C_{E(L/K)}(M)=\CD (L/M)\rtimes G(L/M)=\CD (L/K)^M\rtimes G(L/K)^M
$$ is a bijection with inverse
\begin{eqnarray*}
A\mapsto C_{E(L/K)}(A)&=&C_{L}( A\cap \CD (L/K)_+)\cap L^{A\cap G(L/K)}=\Big(C_{L}( A\cap \CD (L/K)_+)) \Big)^{A\cap G(L/K)}\\
&=& L^{A\cap \CD (L/K)_+}\cap L^{A\cap G(L/K)}=\Big(L^{A\cap \CD (L/K)_+} \Big)^{A\cap G(L/K)}.\\
 \bigg( A=\CD (L/M)&\rtimes &G(L/M)\mapsto 
 L^{\CD (L/M)_+}\cap L^{G(L/M)}=\Big(L^{\CD (L/M)_+} \Big)^{G(L/M}.\bigg)
\end{eqnarray*}

\item $\mL G(L/K)=\Big\{ \Big(\CD (L/M)_+, G(L/M)\Big) \, | \, M\in \CF (L/K)\Big\}$ and for all fields $M\in \CF (L/K)$, 
$$
\Big(\CD (L/M)_+, G(L/M) \Big)=\Big(\CD (L/ML^{gal})_+, G(L/ML^{pi}) \Big)\in \mL (L/L^{gal})\times G(L/L^{pi}).
$$
 In particular, 
$$
\mL G(L/K)=\Big\{ (\CG , H)\in \mL (L/L^{gal})\times \CG (G(L/L^{pi}))\, | \, h\CG h^{-1}=\CG\;\; {\rm  for\; all\; elements}\;\; h\in H\}.
$$

\item  {\sc (Analogue of the  Galois correspondence for subfields of a normal  field extension)}

 The map
$$
\CF (L/K)\ra \mL G(L/K),  \;\; M\mapsto  \Big(\CD (L/M)_+, G(L/M) \Big)=\Big(\CD (L/ML^{gal})_+, G(L/ML^{pi}) \Big)
$$ 
 is a bijection with inverse
$ (\CG , H)\mapsto L^\CG\cap L^H=\Big(L^\CG\Big)^H$ and $h(L^\CG)=L^\CG$ for all $h\in H$.

\item For all fields $M\in \CF (L/K)$, $[M:K]|G(L/M)|\dim_K(\CD (L/M))=[L:K]^2$  or, equivalently, $|G(L/M)|\dim_M(\CD (L/M))=[L:M]^2$.

\item The field extension $L/L^{gal}$ is a finite purely inseparable field extension and 
$\mL (L/L^{gal})=\{ \CD (L/N)_+\, | \, N\in \CF (L/L^{gal})\}$.

\end{enumerate}
\end{theorem}

\begin{proof} The theorem follows from Theorem \ref{1Jun25} ($\char (K)=p$) and Theorem \ref{Gal-1Jun25} ($\char (K)=0$). If $\char (K)=0$ then $\CD (L/M)=L$ and the equality 
 \begin{eqnarray*}
[M:K]|G(L/M)|\dim_K(\CD (L/M))&=&[M:K]|G(L/M)|[L:K]=[M:K][L:M][L:K]\\
&=&[L:K][L:K]= [L:K]^2
\end{eqnarray*}
  holds for all fields $M\in \CF (L/K)$. Now, the equality 
$$|G(L/M)|\dim_M(\CD (L/M))=[L:M]^2$$ follows from the equalities $\dim_K (\CD (L/K))=[M:K]\dim_M (\CD (L/M))$ and $[L:K]=[L:M][M:K]$:
\begin{eqnarray*}
|G(L/M)|\dim_M(\CD (L/M))&=& [M:K]^{-2}\bigg([M:K]|G(L/M)|\dim_K(\CD (L/M))\bigg)= [M:K]^{-2}[L:K]^2\\
&=&[L:M]^2. \qedhere
\end{eqnarray*}
  \end{proof}

{\bf Proof of Theorem \ref{Arb-1Jun25}.} Inclusions of fields $L\supseteq F\supseteq M\supseteq K$ yield reverse  inclusions of their centralizers: $C_E(L)\subseteq C_E(F)\subseteq C_E(M)\subseteq C_E(K)$  where $E=E(L/K)$. By   Theorem \ref{Gen-1Jun25}.(1),
  $$
  L\subseteq \CD (L/F)\rtimes G(L/F)\subseteq \CD (L/M)\rtimes G(L/M) \subseteq \CD (L/K)\rtimes G(L/K).
  $$ 
 Therefore, $\CD (L/F)\subseteq \CD (L/M)\subseteq \CD (L/K)$, 
 $$
   \CD (L/F)_+\subseteq \CD (L/M)_+\subseteq \CD (L/K)_+\;\; {\rm and}\;\; G(L/F)\subseteq G(L/M) \subseteq G(L/K).
  $$
Clearly, $\CF (M/K)\subseteq  \CF (L/K)$. Now, the theorem  follows from Theorem \ref{Gen-1Jun25} by  simply  inserting `$F$' in the notation there, reflecting  the inclusions $
   \CD (L/F)_+\subseteq \CD (L/M)_+$  and $G(L/F)\subseteq G(L/M)$.
   
   \hfill $\square$\\

{\bf Normal subfields of   normal finite field extensions.} Let $ \CA (L\rtimes G(L/K), L)$ be the set of subalgebras of the algebra $L\rtimes G(L/K)$ that contain the field $L$. A subalgebra $A$ of $L\rtimes G(L/K)$ is called  $G(L/K){\bf-stable}$ if 
$$
gAg^{-1}=A\;\; {\rm  for\; all\; elements}\;\;g\in G(L/K).
$$ 
Let $ \CA (L\rtimes G(L/K), L, G(L/K))$ be the set of $G(L/K)$-stable subalgebras of the algebra $L\rtimes G(L/K)$ that contain the field $L$. Corollary \ref{a24Mar25}  describes the sets  $\CA (L\rtimes G(L/K), L)$  and $\CA (L\rtimes G(L/K), L,  G(L/K))$.

 \begin{corollary}\label{a24Mar25}
(\cite{AnGaloisTh-NORMAL-Fields})    Let $L/K$ be a finite field extension of  arbitrary characteristic. Then:
\begin{enumerate}

\item  $ \CA (L\rtimes G(L/K), L)=\{L\rtimes H\, | \, H$ is a  subgroup of $G(L/K) \}$.

\item  $ \CA (L\rtimes G(L/K),L,  G(L/K))=\{L\rtimes N\, | \, N$ is a normal subgroup of $G(L/K) \}$.

\end{enumerate}
\end{corollary}

Theorem  \ref{11Jun25} describes normal sunfields of normal finite field extensions. Also it establishes an order reversing bijection between 
 normal subfields of a  normal finite field extension $L/K$ and the set $\CA (E(L/K),L, G(L/K))$ of $G(L/K)$-stable subalgebras of the algebra $E(L/K)$ that contains the field $L$. 
For a finite field extension $L/K$, let $\CN (L/K)$ be the set of normal subfields $N/K$ of $L/K$ and 
 $\CN(G(L/K))$ be the set of normal subgroups of the group $ G(L/K)$.

  \begin{theorem}\label{11Jun25}
  (\cite{AnGaloisTh-NORMAL-Fields}) 
Let $L/K$ be a normal finite field extension of characteristic $p$. Then:

\begin{enumerate}

\item For each field $M\in  \CN (L/K)$, $M=M^{pi}\t M^{gal}$, where $M^{pi}=M\cap L^{pi}
$ and $M^{gal}:=M\cap L^{gal}$, and 
$$
\CN (L/K)=\CF (L^{pi}/K)\t \CN (L^{gal}/K)
:=\{ N\t \G\, | \, N\in \CF (L^{pi}/K),\,  \G\in \CN (L^{gal}/K)\}.
$$
\item $\CA (E(L/K),L, G(L/K))=\Big\{\CD (L^{pi}/N)\t \Big(L^{gal}\rtimes G(L^{gal}/\G ) \Big)= E(L^{pi}/N)\t L^{gal}/ \G) \Big)\, | \,$ $ N\in \CF (L^{pi}/K),\,  \G\in \CN  (L^{gal}/K)  \Big\}$.

\item {\sc (Analogue of the  Galois correspondence for normal  subfields of a normal  field extension)} 

The map
\begin{eqnarray*}
\CN (L/K)&\ra & \CA (E(L/K),L, G(L/K)), \;\; M\mapsto  C_{E(L/K)}(M)=C_{E(L^{pi}/K)}(M^{pi})\t C_{E(L^{gal}/K)}(M^{gal})\\
&=&E(L^{pi}/M^{pi})\t E(L^{gal}/M^{gal})=\CD (L^{pi}/M^{pi})\t \Big(L^{gal}\rtimes G(L^{gal}/M^{gal}) \Big)
\end{eqnarray*}
is a bijection with inverse
\begin{eqnarray*}
 A&\mapsto& C_{E(L/K)}(A)=\Big( L^{pi}\Big)^{A\cap \CD (L^{pi}/K)_+}\t \Big( L^{gal}\Big)^{A\cap G(L^{gal}/K)}.\\
 \bigg(A&=&\CD (L^{pi}/N)\t \Big(L^{gal}\rtimes G(L^{gal}/\G) \Big) \mapsto   \Big( L^{pi}\Big)^{ \CD (L^{pi}/N)_+}\t \Big( L^{gal}\Big)^{ G(L^{gal}/\G)}.\bigg)
\end{eqnarray*}


\item {\sc (Analogue of the  Galois correspondence for normal  subfields of a normal  field extension)} 

The map
$$
\CN (L/K)\ra \mL (L^{pi}/K)\times \CN (G(L^{gal}/K)), \;\; M=M^{pi}\t M^{gal}\mapsto \Big(\CD (L^{pi}/M^{pi})_+, G(L^{gal}/M^{gal})\Big)
$$
 is a bijection with inverse $(\CG, H)\mapsto \Big(L^{pi} \Big)^\CG\t \Big( L^{gal}\Big)^H$.

\end{enumerate}
\end{theorem}

{\bf Proof of Theorem \ref{Arb-11Jun25}.}  
1.  For each field $M\in  \CN (F/K)$, $M=M^{pi}\t M^{gal}$. It follows from the inclusions 
 $$
 M\subseteq F^{nor}=F^{pi}\t F^{gal}\subseteq L=L^{pi}\t L^{gal}
 $$
  that 
  $M^{pi}=M\cap F^{pi}
=M\cap L^{pi}
$ and $M^{gal}=M\cap F^{gal}=M\cap L^{gal}$. Then 
$$\CN (F/K)=\CN (F^{nor}/K)=\CF \Big(F^{pi}/K\Big)\t \CN  \Big(F^{gal}/K\Big)
:=\Big\{ N\t \G\, | \, N\in \CF  \Big(F^{pi}/K\Big),\,  \G\in \CN  \Big(F^{gal}/K\Big)\Big\}.
$$
 
2--5. Notice that 
$$\CA (E(L/K),L, F, G(L/K))=\{ A\in \, \CA (E(L/K),L,  G(L/K)) | \, E(L/F)\subseteq A\}. 
$$
Suppose that $\char (K)=p$. Now, statements 2--5 follows from statements 2--5 of Theorem \ref{11Jun25}, respectively, simply inserting `$F$' in the notation there reflecting the inclusion $M\subseteq F$.

Suppose that $\char (K)=0$. Then $L^{pi}=K$ and $\CD (L^{pi}/K)=\CD (K/K)=K$, and the theorem reduces  to Theorem  \ref{Gal-11Jun25} which is essentially the Galois Theorem for Galois subfields of a Galois finite  field extension.
 \hfill $\square$\\

{\bf Proof of Theorem \ref{NEWArb-11Jun25}. }  
  Statements 1--3 follow from the equality $\CN (F/K)=\CN (F^{nor}/K)$  and statements 2--4 of Theorem \ref{11Jun25}, respectively. \hfill $\square$\\

{\bf Proof of Theorem \ref{SepNEWArb-11Jun25}.}  
 Suppose that $L/K$ is a separable finite field extension and  so the theorem reduces  to Theorem  \ref{Gal-11Jun25}.
\hfill $\square$





\section{Analogue of the Galois Theory for  separable finite field extensions} \label{SEPARABLE} 

Corollary \ref{Sep-Arb-1Jun25}.(2) is an analogue of the  Galois correspondence for subfields of a separable finite  field extension. Theorem \ref{Gal-11Jun25}.(3) is the  Galois correspondence for Galois  subfields of a separable finite field extension.\\

{\bf Characterization of finite feild extensions $L/K$ with  $ \CD (L/K)=L$.}  Notice that $L\subseteq \CD (L/K)$. Proposition \ref{VB-aC24Mar25} is a criterion for  $ \CD (L/K)=L$.

\begin{proposition}\label{VB-aC24Mar25}
(\cite{AnGaloisTh-NORMAL-Fields})   Let  $L/K$ be a finite  field extension over the field  $K$  of arbitrary  characteristic.  Then the following statements are equivalent:

\begin{enumerate}

\item $\CD (L/K)=L$.

\item $\Der(L/K)=0$.

\item  $\O_{L/K}=0$.

\item  $L/K$ is a separable field extension  ($\Leftrightarrow$ either $\char (K)=0$  or  $\char (K)=p$ and $L=L^{sep}L^p$). 
 
 \item $L_{dif}=L$.
 
 \item The algebra $\CD (L/K)$ is a commutative algebra.

\end{enumerate}
\end{proposition}

{\bf Analogue of the  Galois Theory for  separable  finite field extensions.} Let $F/K$ be a  separable  finite field extension of arbitrary characteristic  and $L/K$ be its normal closure.
Recall that  
$$ \CA (E(L/K),L, F)=\{A\in  \CA (E(L/K),L)\, | \, C_{E(L/K)}(F)\subseteq A\}.$$

 \begin{corollary}\label{Sep-Arb-1Jun25}
Let $F/K$ be a separable finite field extension  and $L/K$ be its normal closure. Then:
\begin{enumerate}
\item $\CA (E(L/K),L, F)=\{ C_{E(L/K)}(M)= L\rtimes G(L/K)^M=L\rtimes G(L/M)\, | \, M\in \CF (L/K)\}$.

\item {\sc (Analogue of the  Galois correspondence for subfields of a separable  field extension)} The map
 $$ 
 \CF (F/K)\ra  \CA (E(L/K), L, F), \;\; M\mapsto  C_{E(L/K)}(M)=L\rtimes G(L/M)=L\rtimes G(L/K)^M
 $$ 
is a bijection with inverse
\begin{eqnarray*}
A\mapsto C_{E(L/K)}(A)&=& L^{A\cap G(L/K)}.\\
 \bigg( A=L\rtimes G(L/M)&\mapsto &
 L^{G(L/M)}.\bigg)
\end{eqnarray*}

\item   For all fields $M\in \CF (F/K)$, $[M:K]|G(L/M)|=[L:K]$.  
\end{enumerate}

\end{corollary}

\begin{proof} 1 and 2. Statements 1 and 2  follow from Theorem \ref{Arb-1Jun25}.(1,2) and the fact  that for the  separable finite field extension $L/K$, $\CD (L/K)=L$ (Proposition \ref{VB-aC24Mar25}) and $\CD (L/K)_+=0$. 

3. Statement 3  follows from Theorem \ref{Arb-1Jun25}.(3) and the equality $\CD (L/K)=L$:  For all fields $M\in \CF (F/K)$, 
$$
[L:K]^2=[M:K]|G(L/M)|\dim_K(\CD (L/M))=[M:K]|G(L/M)[L:K],
$$ 
and so $[M:K]|G(L/M)|=[L:K]$.  
\end{proof}

{\bf Galois subfields of  separable finite field extensions.}  Theorem  \ref{Gal-11Jun25} gives  an order reversing bijection between 
 Galois subfields of a   Galois  finite field extension $L/K$ and the set $\CA (E(L/K),L, G(L/K))$ of $G(L/K)$-stable subalgebras of the algebra $E(L/K)$ that contain the field $L$. Recall that (Corollary \ref{a24Mar25}.(2)), 
 $$
 \CA (E(L/K),L, G(L/K))=\{L\rtimes  N \, | \, N\;\;{\rm is\; a\; normal\; subgroup\; of}\;\; G(L/K) \}.
 $$
For a finite field extension $L/K$, let $\CG (L/K)$ be the set of  Galois  subfields $N/K$ of $L/K$ and 
 $\CN(G(L/K))$ be the set of normal subgroups of the group $ G(L/K)$. Recall that the set $\CG (L/K)$ is the set of all $K$-subfields of $L$ that are  G-extensions (Proposition \ref{A8May25}). 
 Theorem \ref{11Jun25}.(3) 
is the  Galois correspondence for Galois  subfields of the Galois field extension $L/K$.

\begin{theorem}\label{Gal-11Jun25}
(\cite{GaloisTh-RingThAp})   Let $L/K$ be a  G-extension (i.e. a Galois extension). Then: 
\begin{enumerate}

\item $\CA (E(L/K),L, G(L/K))=\{C_E(\G )=L\rtimes G(L/\G)\, | \,$ $ \G\in \CG  (L/K)  \}=\{L\rtimes N\, | \, N\in \CN  (G(L/K))  \}$ and $\CN (G(L/K))= \{ G(L/\G)\, | \,$ $ \G\in \CG  (L/K)  \}$.

\item The map
$$
\CG (L/K)\ra  \CA (E(L/K),L, G(L/K)), \;\; \G \mapsto  C_{E}(M)=L\rtimes G(L/\G) 
$$
 is a bijection with inverse
$$ A\mapsto C_{E}(A)=L^{A\cap G(L/K)}.
$$
\item {\sc (The  Galois correspondence for Galois  subfields of a Galois field extension)} 

The map
$$
\CG (L/K)\ra \CN (G(L/K)), \;\; \G\mapsto  G(L/\G)=G(L/K)^{\G}
$$ 
is a bijection with inverse $
\G\mapsto L^{\G}.$
\end{enumerate}
\end{theorem}

{\bf Proof of Theorem \ref{GalArb-11Jun25}.}  
 The  theorem     follows from 
 Theorem \ref{Gal-11Jun25}. \hfill $\square$


\section{Analogue of the Galois Theory for  purely inseparable finite field extensions} \label{PURELY-INSEP} 

The results presented in this section were originally proved in \cite{AnGaloisTh-NORMAL-Fields}. We include them here for the sake of completeness.
 In this section,  $K$ is a field of characteristic $p$ and $L/K$ is a purely inseparable   finite field extension. Theorem \ref{C24Mar25} provides an explicit description of the algebra of differential operators  $\CD (L/K)$. Theorem \ref{D24Mar25} is an analogue of the Galois Theory for purely inseparable field extensions.\\ 

 
{\bf Description of the algebra of differential operators on a purely inseparable finite field extension.}   Theorem \ref{C24Mar25} provides an explicit description of the algebra of differential operators 
 $\CD (L/K)$ for a purely inseparable finite field extension $L/K$ in characteristic $p$. 

\begin{theorem}\label{C24Mar25}
Let $K$ be a field of  characteristic $p$, 
 $L/K$ be a  purely inseparable field extension  of degree $n:=[L:K]<\infty$. Then:
 \begin{enumerate}
 
 \item $\CD (L/K)=E(L/K)\simeq M_n(K)\in \fCK$ and $\Deg (\CD (L/K))=n$. In particular, all purely inseparable finite field extensions are B-extensions.
 
 \item $L\in \Max_s (\CD (L/K))$.

 \item $Z(\CD (L/K))=K$.
 
 \item $L$ is the only (up to isomorphism) simple 
 $\CD (L/K)$-module.
 
 \item $\End_{\CD (L/K)}(L)=K$.

\end{enumerate}
\end{theorem}

Let $\CA ( \CD (L/K), L, sim):=\{ A\in \CA ( \CD (L/K), L)\, | \, A$ is a simple algebra$\}$. 
Proposition \ref{aD24Mar25} shows that all  subalgebras of $\CD (L/K)$ that contain the field $L$ are simple algebras where $L/K$ is a  purely inseparable field extension  of finite degree.

\begin{proposition}\label{aD24Mar25}
 Let $K$ be a field of  characteristic $p$, 
 $L/K$ be a  purely inseparable field extension  of finite degree. Then 
 $$
 \CA ( \CD (L/K), L)=\CA ( \CD (L/K), L, sim),
 $$
  i.e. every subalgebra of $\CD (L/K)$ that contains the field $L$ is a simple algebra.
\end{proposition}

{\bf Analogue of the Galois Theory for purely inseparable field extensions.}  
Theorem \ref{D24Mar25} is an analogue of the Galois Theory for purely inseparable finite field extensions that gives a bijection between subfields and the algebras  of   differential operators (Theorem \ref{D24Mar25}.(1,2)) and between subfields and the dominant Lie algebras  (Theorem \ref{D24Mar25}.(3,4)).



\begin{theorem}\label{D24Mar25}
 Let $K$ be a field of  characteristic $p$, 
 $L/K$ be a  purely inseparable finite  field extension. Then:
 \begin{enumerate}

\item $\CA ( \CD (L/K), L)\stackrel{{\rm Pr.}\, \ref{aD24Mar25}}{=}\CA ( \CD (L/K), L, sim)=\{ \CD(L/M)\, | \, M\in \CF (L/K) \}$.

\item {\sc (Analogue of the  Galois correspondence for subfields of purely inseparable  field extensions)} The map 
$$\CF (L/K)\ra \CA ( \CD (L/K), L)\stackrel{{\rm Pr.}\, \ref{aD24Mar25}}{=}\CA (\CD (L/K), L, sim), \;\; M\mapsto  C_{\CD (L/K)}(M)=\CD (L/M)$$  is a bijection with inverse $A\mapsto C_{\CD (L/K)}(A)$.

\item $\mL (L/K)=\{ \CG (M)=\CD(L/M)_+\, | \, M\in \CF (L/K)\}$.

\item  {\sc (Analogue of the Galois correspondence  for subfields of a purely inseparable finite  field extension)}

{\em The map}
$$
\CF (L/K)\ra  \mL (L/K), \;\; M\mapsto \CG (M)=\CD(L/M)_+
$$
{\em is a bijection with inverse} $\CG\mapsto L^\CG=C_L(\CG)$.

\end{enumerate}
\end{theorem}

Recall that  the classical Galois correspondence for finite Galois field extensions is the bijection $M\mapsto G(L/M)$. So, in Theorem \ref{D24Mar25}   the algebras differential operators $\CD (L/K)$ on purely inseparable finite field extensions $L/K$ or the dominant Lie algebras  $\mL(L/K)$ play the role of the Galois group in the classical Galois Theory.


 








\section{Examples of Galois type correspondences and descriptions of  $L^{pi}$,   $L^{gal}$ and  $G(L/K)$ of a normal finite field extension $L/K$} \label{EXAMPLE} 

In this section, we consider  an example  of the Galois correspondences in Theorem \ref{Arb-1Jun25}.(2) and Theorem \ref{Arb-11Jun25}.(2), see Proposition \ref{XA7Jul25}.  It also provides an example of a normal finite field extension in prime characteristic such that  all subfields are normal. 
For a normal finite field extension $L/K$,  Proposition \ref{Ext-A12Apr25}   provides explicit descriptions of the subfields $L^{pi}$ and  $L^{gal}$ of $L/K$ and the automorphism group $G(L/K)$.

 Recall that  that every normal finite field extension $L/K$ in  characteristic $p$ is the tensor product $L=L^{pi}\t_K L^{sep}$ of its subfields $L^{pi}$ and $L^{sep}=L^{gal}$. This is not the case if we drop the assumption of normality.

\begin{lemma}\label{16Nov25}
 A finite field extension $L/K$ is a normal extension iff $L=L^{pi}\t L^{gal}$. So, a finite field extension $L/K$ is not normal  iff $L\neq L^{pi}\t L^{gal}$ iff $[L:K]> [L^{pi}:K][ L^{gal}:K]$. 
\end{lemma}

\begin{proof} $(\Rightarrow)$ The implication is known.

$(\Leftarrow)$ Since $L=L^{pi}\t L^{gal}$ and the field extensions $L^{pi}/K$ and  $L^{gal}/K$ are normal, the finite field extension $L/K$ is  normal.
\end{proof}

\begin{theorem}\label{23May25}
 (\cite{AnGaloisTh-NORMAL-Fields})    Let $L/K$ be a finite field extension. Then $L_{dif}=L^{sep}=L^{\CD (L/K)_+}$.
\end{theorem}

{\bf Explicit descriptions of the subfields $L^{pi}$ and $L^{gal}$ of a normal finite field extension $L/K$ and the group $G(L/K)$.} 
\begin{definition}
Suppose that $K$ is a field of prime characteristic $p>0$. Then each non-scalar polynomial $f(x)\in K[x]$ admits a unique presentation 

\begin{equation}\label{f=fsepxn}
 f(x)=f^{sep}(x^{p^n})\;\; {\rm where}\;\; f^{sep}(x)\in K[x]\;\; {\rm is\; a\; separable\; polynomial\; and}\;\; n\geq 0.
\end{equation}
The equality (\ref{f=fsepxn}) is called a {\bf separable presentation} of the polynomial $f(x)$. The polynomial $f^{sep}(x)$ is called the {\bf separable part} of $f$ and the natural number $n$ is called the {\bf inseparability degree} of $f(x)$ and denoted by $\deg_{ins}(f)$. 
\end{definition}

For the polynomial $f(x)=\sum_{i\geq 0} \mu_ix^i$,  ${\rm coef}(f):= \{\mu_i\, | \, i\geq 0\}$ is the 
  set of its coefficients. Clearly, 
\begin{equation}\label{f=fsepxn-2}
 {\rm coef}(f)={\rm coef}(f^{sep}).
\end{equation}
Notice that 
\begin{equation}\label{f=fsepxn-1}
\deg(f)=p^n\deg (f^{sep})\;\; {\rm where}\;\; n=\deg_{ins}(f).
\end{equation}
For the polynomial $f(x)$ as in (\ref{f=fsepxn}),   $f(x)=\sum_{i\geq 0} \l_ix^{ip^n}=\bigg(\sum_{i\geq 0} \l_i^\frac{1}{p^n} x^i\bigg)^{p^n}$, where $n=\deg_{ins}(f)$, and so 

\begin{equation}\label{f=fsepxn-3}
f(x)=\bigg( {f^{sep}}^\frac{1}{p^n}\bigg)^{p^n}\;\; {\rm where}\;\;  {f^{sep}}^\frac{1}{p^n}:=\sum_{i\geq 0} \l_i^\frac{1}{p^n} x^i\in K\Big({\rm coef}(f)^\frac{1}{p^n}\Big)[x]
\end{equation}
is a {\em separable} polynomial over the field $K\Big({\rm coef}(f)^\frac{1}{p^n}\Big)$. Clearly,
\begin{equation}\label{f=fsepxn-4}
{\rm roots}(f)={\rm roots}( {f^{sep}}^\frac{1}{p^n}).
\end{equation}

\begin{proposition}\label{Ext-A12Apr25}
Suppose that $K$ is a field of prime characteristic $p>0$, $L/K$ is a normal finite field extension, i.e. it is a splitting field of  a set of polynomials $\{ f_i(x)=f_i^{sep}(x^{p^{n_i}})\, | \,i\in I\}$, and $L^{pi}/K$ is the largest purely inseparable field extension in the field $L/K$. Then: 
\begin{enumerate}

\item  $L^{pi}=K\bigg(\bigcup_{i\in I} {\rm coef}(f_i)^\frac{1}{p^{n_i}} \bigg)$ where ${\rm coef}(f_i):=\{ \l^\frac{1}{p^{n_i}}\, | \, \l \in 
{\rm coef}(f_i)\}$.

\item $L/L^{pi}$ is a Galois field extension which is generated by the roots of the polynomials  $\{ f_i\, | \, i\in I\}$ or, equivalently, by the roots of separable polynomials $\{ {f_i^{sep}}^\frac{1}{p^{n_i}})\, | \, i\in I\}\subseteq L^{pi}[x]$ (this follows from (\ref{f=fsepxn-4})).

\item  $L^{gal}=K\bigg(\bigcup_{i\in I} {\rm roots}( f_i^{sep}) \bigg)$.

\item  $L/L^{gal}=L^{gal}\bigg(\bigcup_{i\in I} {\rm roots}(f_i)^\frac{1}{p^{n_i}} \bigg)$
 is a purely inseparable field extension.

\item The  field extension $L/L^{pi}$ is a Galois field extension with Galois group $G(L/L^{pi})=G(L/K)=G(L^{gal}/K)$.

\item $L^{pi}=L^{G(L/L^{pi})}=L^{G(L/K)}=L^{G(L^{gal}/K)}$.

\item $L^{gal}=L^{sep}=L^{\CD (L/K)_+}$.

\end{enumerate}
\end{proposition}

\begin{proof} 1, 2 and 5.  Clearly, $L^{pi}\supseteq L^{pi}_c:=K\bigg(\bigcup_{i\in I} {\rm coef}(f_i)^\frac{1}{p^{n_i}} \bigg)$ and the field extension 
$$
L/L^{pi}_c=L^{pi}_c\bigg(\bigcup_{i\in I} {\rm roots}\Big({f_i^{sep}}^\frac{1}{p^{n_i}} \Big) \bigg)
$$ 
is a Galois field extension as it is a splitting field of the set of separable polynomials $\Big\{ {f_i^{sep}}^\frac{1}{p^{n_i}})\, | \, i\in I\Big\}\subseteq L^{pi}_c[x]$. Since $L_c^{pi}\subseteq L^{pi}\subseteq L$, $L^{pi}/L^{pi}_c$ is a purely inseparable field extension and $L/L_c^{pi}$ is a Galois field extension, we must have $L_c^{pi}=L^{pi}$. So, $L/L^{pi}$ is a Galois field extension. Notice that $G(L/L^{pi})\subseteq G(L/K)$.  The field extension $L^{pi}/K$ is a purely inseparable finite field extension. Therefore, $L^{pi}/K$ is a normal field extension with $G(L^{pi}/K)=\{e\}$. Therefore, 
$G(L/K)= G(L/L^{pi})$.

The finite field extension $L/K$ is normal. By Lemma \ref{16Nov25}, $L=L^{pi}\t L^{gal}$. Therefore, $G(L/L^{pi})=G(L^{gal}/K)$.

3 and 4. Clearly, $L^{gal}\supseteq L^{gal}_r:= K\bigg(\bigcup_{i\in I} {\rm roots}( f_i^{sep}) \bigg)$ and the field extension $L^{gal}_r/K$ is a Galois finite field extension. Notice that  
$$
L/L^{gal}_r=L^{gal}_r\bigg(\bigcup_{i\in I} {\rm roots}(f_i)^\frac{1}{p^{n_i}} \bigg).
$$ 
Therefore $L^{gal}= L^{gal}_r$. Hence, 
$L^{gal}=K\bigg(\bigcup_{i\in I} {\rm roots}( f_i^{sep}) \bigg)$ and   $L/L^{gal}=L^{gal}\bigg(\bigcup_{i\in I} {\rm roots}(f_i)^\frac{1}{p^{n_i}} \bigg)$
 is a purely inseparable field extension.

6. By statement 3, the  finite field extension $L/L^{pi}$ is a Galois field extension with Galois group $G(L/L^{pi})=G(L/K)$. By the Galois Theorem,  $L^{pi}=L^{G(L/L^{pi})}=L^{G(L/K)}=L^{G(L^{gal}/K)}$.

7. The field extension $L/K$ is a normal finite field extension. Therefore, $L^{gal}=L^{sep}$.  By Theorem \ref{23May25},  $L^{sep}=L^{\CD (L/K)_+}$.  \end{proof}

{\bf Examples  of the Galois correspondences in Theorem \ref{Arb-1Jun25}.(2) and Theorem \ref{Arb-11Jun25}.(2).}
For a prime number $p>0$, let $\mF_p:=\Z / p\Z$. 

\begin{proposition}\label{XA7Jul25}
Let $p\neq 2$ be a prime number and  $K=\mF_p(x,y)$ be a field of rational functions in the variables $x$ and $y$, and  $L=K(y^\frac{1}{p},x^\frac{1}{2})$.  Then: 
\begin{enumerate}

\item The finite field extension $L/K$ is  normal, $L=L^{pi}\t L^{gal}$ where $L^{pi}= K(y^\frac{1}{p})$ and $L^{gal}=K(x^\frac{1}{2})$.

\item $\CF (L/K)=\CN (L/K)=\{ K, L^{pi},  L^{gal}, L\}$, i.e. all subfields of $L/K$ are normal. Furthermore,  $\CF (L/K)=\{ M^{pi}\t M^{gal}\, | \,  M^{pi}\in \CF (L^{pi}),  M^{gal}\in \CF (L^{gal})\}$ where $\CF (L^{pi})=\{K, L^{pi} \}$ and   $\CF (L^{gal})=\{ K, L^{gal}\}$.

\item {\sc (Analogue of the  Galois correspondence for (normal)  subfields of $L/K$)} 

The map
$$
\CF (L/K)=\CN (F/K)\ra  \CA (E(L/K),L, F, G(L/K)), \;\; M\mapsto \CD (L^{pi}/M^{pi})\t \Big(L^{gal}\rtimes G(L^{gal}/M^{gal}) \Big)
$$
is a bijection where 

$$
\CD (L^{pi}/M^{pi})=\begin{cases}
\CD (L^{pi}/K)=\bigoplus_{j=0}^{p-1}L^{pi}\der^j& \text{if }M^{pi}=K,\\
L^{pi}& \text{if }M^{pi}=L^{pi},\\
\end{cases}$$
 and
 $$ 
G (L^{gal}/M^{gal})=\begin{cases}
G(L^{gal}/K)=\{ e,\s\}& \text{if }M^{gal}=K,\\
\{ e\}& \text{if }M^{gal}=L^{gal},\\
\end{cases}
$$
where $\der:=\frac{\der}{\der y^\frac{1}{p}}\in \Der_K(K(y^\frac{1}{p}))$ and $\s (x^\frac{1}{2})=-x^\frac{1}{2}$.

\item {\sc (Analogue of the  Galois correspondence for normal  subfields of a normal  field extension)} 

The map
$$
\CN (L/K)\ra \mL (L^{pi}/K)\times \CN (G(L^{gal}/K)), \;\; M=M^{pi}\t M^{gal}\mapsto \Big(\CD (L^{pi}/M^{pi})_+, G(L^{gal}/M^{gal})\Big)
$$
 is a bijection with inverse $(\CG, H)\mapsto \Big(L^{pi} \Big)^\CG\t \Big( L^{gal}\Big)^H$ where
 $$
\CD (L^{pi}/M^{pi})_+=\begin{cases}
\CD (L^{pi}/K)_+=\bigoplus_{j=1}^{p-1}L^{pi}\der^j& \text{if }M^{pi}=K,\\
\{0\}& \text{if }M^{pi}=L^{pi},\\
\end{cases}
$$
and 
$$ 
G (L^{gal}/M^{gal})=\begin{cases}
G(L^{gal}/K)=\{ e,\s\}& \text{if }M^{gal}=K,\\
\{ e\}& \text{if }M^{gal}=L^{gal}.\\
\end{cases}
$$
\end{enumerate}
\end{proposition}

\begin{proof} 1 and 2. Clearly, $L=K(y^\frac{1}{p})\t K(x^\frac{1}{2})$ where  $K(y^\frac{1}{p})/K$ is a purely inseparable finite field extension   and 
$K(x^\frac{1}{2})$ is a Galois finite field extension with Galois group $G(K(x^\frac{1}{2})/K)=\{ e, \s\}$ where $\s (x^\frac{1}{2})=-x^\frac{1}{2}$.
 Therefore $L=L^{pi}\t L^{gal}$ is a normal finite field extension where $L^{pi}=K(y^\frac{1}{p})$ and $ L^{gal}=K(x^\frac{1}{2})$. Since the degrees $[L^{pi}:K]=p$ and  $[L^{gal}:K]=2$ are prime numbers, $\CF (L^{pi})=\{K, L^{pi} \}$ and   $\CF (L^{gal})=\{ K, L^{gal}\}=\CN (L^{gal})$.
Therefore, 
$$
\CN (L/K)=\{ M^{pi}\t M^{gal}\, | \,  M^{pi}\in \CF (L^{pi}),  M^{gal}\in \CF (L^{gal})\}=\{ K, L^{pi},  L^{gal}, L\}.
$$
 To finish the proof we have to prove that $\CF (L/K)=\CN (L/K)$. Let $M$ be a proper subfield 
of the field extension $L/K$. We have to show that $M/K$ is a normal field extension. Since
 $$[L:K]=[L^{pi}\t L^{gal}:K]=[L^{pi}:K][L^{gal}:K]=2p$$
 and 2 and $p$ are distinct prime numbers and $[M:K] | 2p$, there are two options either $[M:K]=2$ or otherwise $[M:K]=p$. In both cases the degree $[M:K]$ is a prime number. Therefore, the field extension $M/K$ is generated by a single element, say $a\in L\backslash K$, i.e. $M=K(a)$. Since $L=L^{pi}\oplus L^{pi}x^\frac{1}{2}$, we have that $a=\alpha +\beta x^\frac{1}{2}$ for some elements $\alpha, \beta \in L^{pi}$.
 
 Suppose that $\beta =0$. Then the element $a=\beta \in L^{pi}\backslash K$, and  so $M=K(a)=L^{pi}\in \CN (L/K)$.
 
Suppose that $\beta \neq 0$. Then 
$$x^\frac{p}{2}=x^\frac{p-1}{2}x^\frac{1}{2}\;\; {\rm and}\;\; x^\frac{p-1}{2}\in K.$$ 
Since 
 $$ M\ni a^p=\alpha^p+\beta^px^\frac{p}{2}=\alpha^p+\beta^px^\frac{p-1}{2}x^\frac{1}{2}, $$
  $\alpha^p, \beta^px^\frac{p-1}{2}\in K$ and $\beta^px^\frac{p-1}{2}\neq 0$, we must have $x^\frac{1}{2}\in M$. Therefore, $M=K(x^\frac{1}{2})=L^{gal}\in \CN (L/K)$.

3.  Recall that $G(L^{gal}/K)=\{ e,\s\}$. It is well-known and easy to prove that $ \CD (L^{pi}/K)=\bigoplus_{j=0}^{p-1}L^{pi}\der^j$:
$$ S:=\sum_{j=1}^{p-1}L^{pi}\der^j=\bigoplus_{j=1}^{p-1}L^{pi}\der^j\subseteq \CD(L^{pi}/K)\subseteq E(L^{pi}/K)$$ and $\dim_K(S)=p^2=\dim_K ( E(L^{pi}/K))$, we deduce that $S=\CD (L^{pi}/K) =E(L^{pi}/K)$.
 Now,  statement 3 follows from Theorem \ref{Arb-1Jun25}.(2) and Theorem \ref{Arb-11Jun25}.(2).
 
 4. Statement 4 follows from statement 3 and Theorem \ref{11Jun25}.(4).
\end{proof}

 {\bf Licence.} For the purpose of open access, the author has applied a Creative Commons Attribution (CC BY) licence to any Author Accepted Manuscript version arising from this submission.

{\bf Declaration of interests.} The authors declare that they have no known competing financial interests or personal relationships that could have appeared to influence the work reported in this paper.


{\bf Data availability statement.} Data sharing not applicable – no new data generated.

{\bf Funding.} This research received no specific grant from any funding agency in the public, commercial, or not-for-profit sectors.

\small{

School of Mathematical  and Physical Sciences

Division of Mathematics

University of Sheffield

Hicks Building

Sheffield S3 7RH

UK

email: v.bavula@sheffield.ac.uk}

\end{document}